\documentclass[12pt]{article}
\usepackage{amsmath}
\usepackage{bbding}
\usepackage{cases}
\usepackage{amsfonts}
\usepackage{mathrsfs}
\usepackage{amssymb}
\usepackage{graphicx}
\usepackage{enumerate}
\usepackage{datetime}
\usepackage{color}
\usepackage [top=0.8in, bottom=1in, left=2cm, right=2cm]{geometry}
\newtheorem{theorem}{Theorem}[section]
\newtheorem{definition}[theorem]{Definition}
\newtheorem{algorithm}[theorem]{Algorithm}
\newtheorem{lemma}[theorem]{Lemma}

\newtheorem{proposition}[theorem]{Proposition}
\newtheorem{example}[theorem]{Example}
\newtheorem{remark}[theorem]{Remark}
\newenvironment{proof}{\noindent {\bf Proof.\ }}{$\blacksquare$\vspace{2ex}}
\newcommand{\lt}{\operatorname{lt} }
\newcommand{\lc}{\operatorname{lc} }
\newcommand{\lm}{\operatorname{lm} }
\newcommand{\Rem}{\operatorname{Rem} }
\newcommand{\Spoly}{\operatorname{Spoly} }
\newcommand{\Svect}{\operatorname{Svect} }

\newcommand{\lclm}{\operatorname{lclm} }
\newcommand{\Supp}{\operatorname{Supp} }
\newcommand{\Irr}{\operatorname{Irr} }

\newcommand{\topp}{\operatorname{top} }
\newcommand{\sh}{\operatorname{sh}}
\newcommand{\HF}{\operatorname{HF} }
\newcommand{\MM}{\mathcal{M} }
\newcommand{\NNN}{\mathcal{N} }

\newcommand{\GK}{\operatorname{GKdim}}

\newcommand{\NN}{\mathbb{N}}
\newcommand{\RR}{\mathbb{R}}

\newcommand{\QQ}{\mathbb{Q}}
\newcommand{\D}{\mathcal{D}}
\newcommand{\s}{\mathcal{S}}

\newcommand{\gbs}{Gr\"obner bases}
\newcommand{\gs}{Gr\"obner-Shirshov}
\newcommand{\gsb}{Gr\"obner-Shirshov basis}
\newcommand{\gsbs}{Gr\"obner-Shirshov bases}
\newcommand{\fg}{finitely generated}

\newcommand{\dd}{differential difference}
\newcommand{\Dd}{Differential difference}
\newcommand{\IFF}{if and only if }
\newcommand{\LQ}[2]{\operatorname{LQ}(#1,#2) }

\newcommand{\vect}[1]{\boldsymbol{#1}}

\newcommand{\tdeg}{\operatorname{tdeg}}

\usepackage{hyperref}
\begin{document}
\title{Gelfand-Kirillov Dimensions of Modules over Differential Difference Algebras}
\author{\small Yang Zhang and
Xiangui Zhao\\
\small Department of Mathematics, University of Manitoba\\
\small Winnipeg, MB R3T 2N2, Canada\\
\small yang.zhang@umanitoba.ca, \ \ xian.zhao@umanitoba.ca
}
\date{}

\maketitle

\begin{abstract}
Differential difference algebras are generalizations of polynomial algebras, quantum planes, and Ore extensions of automorphism type and of derivation type. In this paper, we investigate the Gelfand-Kirillov dimension of a finitely generated module over a differential difference algebra through a computational method: Gr\"obner-Shirshov basis method. We develop the Gr\"obner-Shirshov basis theory of differential difference algebras, and of finitely generated modules over differential difference algebras, respectively. Then, via Gr\"obner-Shirshov bases, we give algorithms for computing the Gelfand-Kirillov dimensions of cyclic modules and finitely generated modules over differential difference algebras.
\end{abstract}



\section{Introduction}
Let $k$ be a field, $A$ be an associative $k$-algebra with identity $1$, and $M$ be a left $A$-module. Then the \emph{Gelfand-Kirillov dimension} of $M$ (\cite{Krause-Lenagan_2000}, Chapter 5) is defined  by
\[
\GK(M)=\sup_{V,F}\overline{\lim_{n\to\infty}}\log_n\dim_k(V^nF)
\]
where the supremum is taken over all finite dimensional subspaces $V$ of $A$ containing $1$ and all finite dimensional subspaces $F$ of $M$. Gelfand-Kirillov dimension is a very useful tool for investigating modules over noncommutative algebras. Basic properties and applications of Gelfand-Kirillov dimension can be found in \cite{Krause-Lenagan_2000}.

\gsb\   theory  is a powerful computational tool for both commutative and noncommutative algebras (see the survey \cite{BokutChenShumSurvey2010}, and more algebraic structures which admit \gsb\ theory can be found in, for example, dialgebras \cite{BokutChenLiu2010Dialgebra}, matabelian Lie algebras \cite{ChenChen2012Matabelian}, $L$-algebras \cite{BokutChenHuang2013Lalgebra}, semirings \cite{BokutChenMo2013Semiring}). For commutative algebras, the dimension of an algebraic variety can be efficiently computed by using \gsbs\ to compute the growth of the Hilbert function (or Hilbert polynomial)(see \cite{Becker-Weispfenning_1993}). The Gelfand-Kirillov dimension of a \fg\ module over a \fg\ algebra is also closely related to Hilbert function and thus it is possible to compute it  for some specific classes of noncommutative algebras by using \gsbs. For example, Bueso et al. \cite{Bueso+Jimenez+Jara_1997} computed the Gelfand-Kirillov dimension of a cyclic module over an almost commutative algebra, Torrecillas \cite{torrecillas1999gelfand} considered the Gelfand-Kirillov dimension of \fg\ graded modules over multi-graded \fg\ algebas, Li and Wu \cite{li2002noncommutative} extended this method to cyclic modules over solvable polynomial algebras (also known as PBW algebras), and Bueso et al. \cite{Bueso2001computing} extended it to \fg\ modules over PBW algebras.

Differential difference algebras were first defined by Mansfield and Szanto in \cite{mansfield2003elimination},
which arose from the calculation of symmetries of discrete systems. Differential difference algebras are generalizations of several classes of (skew) polynomial rings/algebras, e.g., commutative polynomial algebras, skew polynomials of derivation/automorphism type (\cite{McConnell-Robson2001}, Chapter 1) and the quantum plane (\cite{kassel1995quantum}, Chapter IV).
Mansfield and Szanto \cite{mansfield2003elimination} developed the \gsb\ theory (where they use the term \gbs\ instead) of \dd\ algebras by using a special kind of left admissible orderings,
which they called differential difference orderings.
In this paper, we  generalize the \gsb\ theory of \dd\ algebras to any left admissible ordering and develop the \gsb\ theory of \fg\ free modules over \dd\ algebras. By using the theory we develop in this paper, we compute the Gelfand-Kirillov dimensions of \fg\ modules over \dd\ algebras.

This paper is organized as follows. We give the definition and properties of \dd\ algebras in Section 2.  In Section 3, we generalize the main results of Mansfield and Szanto \cite{mansfield2003elimination} on \gsbs\ to \dd\ algebras with respect to \dd\ monomial orderings (see Definition \ref{Def_DD-ordering}). Then, in Section 4, we apply the theory we develop to compute the Gelfand-Kirillov dimension of a cyclic module over a \dd\ algebra. We develop the \gsb\ theory of \fg\ modules over \dd\ algebras in Section 5. Finally we investigate the Gelfand-Kirillov dimension of \fg\ modules over \dd\ algebras in Section 6.

\section{Preliminaries}

Throughout this paper, we assume that  $k$  is a field with characteristic $0$ and  all algebras  are unital  associative $k$-algebras.
A mapping $\delta$  on an $k$-algebra $A$ is called a \emph{$k$-derivation} (or only {\it derivation} for short) on $A$ provided that, for any $a,b\in A$ and $c\in k$, $\delta(ca+b)=c\delta(a)+\delta(b)$ and
$\delta(ab)=a\delta(b)+\delta(a)b$.
If $R\subseteq A$ is a subalgebra and $a_1,\ldots,a_p\in A$, where $p$ is a positive integer, then $R\langle a_1,\ldots, a_p\rangle$ denotes the subalgebra of $A$ generated by $R$ and $\{a_1,\ldots,a_p\}$.

First we recall the definition of \dd\ algebras.
\begin{definition}(cf., \cite{mansfield2003elimination})\label{Def_DD_alg}
An algebra $A$ is called \emph{a \dd\ algebra} of type $(m,n)$, $m,n\in\NN$, over a subalgebra $R\subseteq A$ if there exist elements
$S_1,\ldots,S_m,D_1,\ldots,D_n$ in $A$ such that
\begin{enumerate}[(i)]
\item the set
$\{S_1^{\alpha_1}\cdots   S_m^{\alpha_m}D_1^{\beta_1}\cdots D_n^{\beta_n}:\alpha_1,\ldots,\alpha_m,\beta_1,\ldots,\beta_n\in \NN\}$
 forms a basis for $A$ as a free left $R$-module.

\item $D_ir=rD_i+\delta_i(r)$ for any $1\leq i\leq n$ and $r\in R$, where
$\delta_i$ is a derivation on $R$.
\item $S_ir=\sigma_i^{-1}(r)S_i$, {or equivalently $rS_i=S_i\sigma_i(r)$}, for any $1\leq i\leq m$ and $r\in R$, where
 $\sigma_i$ is a $k$-algebra automorphism on the subalgebra $R\langle D_1, \dots, D_n \rangle\subseteq A$ such that  the restriction $\sigma_i|_R$ is a $k$-algebra automorphism on $R$ and $\sigma_i(D_j)=\sum_{l=1}^n a_{ijl}D_l,\ \ a_{ijl}\in R.$
\item $D_iS_j=S_j\sigma_j (D_i)$, $1\leq i\leq n,1\leq j\leq m$.
\item $S_iS_j=S_jS_i$, $ 1\leq i,j\leq m$; \ $D_{i'}D_{j'}=D_{j'}D_{i'}$, $1\leq i',j'\leq n$.
\item For any $1\leq i,j\leq n$ and $1\leq i',j'\leq m$, the composition
$
\delta_i\circ\delta_j=\delta_j\circ\delta_i$ and $
\sigma_{i'}\circ\sigma_{j'}=\sigma_{j'}\circ\sigma_{i'}.
$
\end{enumerate}
\end{definition}

{Let $\mathcal{D} = \{D_1, \dots, D_n\}$ and $\mathcal{S} = \{S_1, \dots, S_m\}$.} If $A$ is a \dd\ algebra over $R$ as defined above, we denote $A=R[\s,\D;\sigma,\delta]$. Denote $\s^{\alpha}\D^{\beta}=S_1^{\alpha_1}\cdots S_m^{\alpha_m}D_1^{\beta_1}\cdots D_n^{\beta_n}$ for $\alpha=(\alpha_1,\dots,\alpha_m)\in \NN^m$, $\beta=(\beta_1,\ldots,\beta_n)\in\NN^n$.

\begin{remark}
\label{rem_MonomialForm}\upshape
\Dd\ algebras can be also defined as iterated Ore extensions:
$$R[D_1;\delta_1] \cdots [D_n;\delta_n][S_1; \sigma_1]\cdots [S_m; \sigma_m]$$
with careful choices of $\sigma$ and $\delta$ to get conditions (v) and (vi) in the above definition. In the language of iterated Ore extensions, it is natural to take the set  $\{\D^{\beta}\s^{\alpha}:\alpha\in\NN^m, \beta\in\NN^n\}$ as a standard $R$-basis of $A$.
However, in Definition \ref{Def_DD_alg}, we take $\{\s^{\alpha}\D^{\beta}:\alpha\in\NN^m, \beta\in\NN^n\}$ rather than $\{\D^{\beta}\s^{\alpha}:\alpha\in\NN^m, \beta\in\NN^n\}$ as a standard $R$-basis of $A$ since the former has more advantages related to computational properties of $A$. For example, the usual degree-lexicographical ordering (Example \ref{Exam_DD-monomial_order}) works well as a left admissible ordering (which is essential to develop \gsb\ theory of $A$) with $\s^{\alpha}\D^{\beta}$, but not with $\D^{\beta}\s^{\alpha}$.


\end{remark}

Several well-known  classes of skew polynomial algebras are contained in the class of \dd\ algebras. But on the other hand,  a \dd \ algebra is not necessarily contained in some well-known classes of noncommutative algebras, for example,  algebras of solvable type \cite{Weispfenning1990non},  PBW extensions \cite{bell1988uniform}, and  G-algebras \cite{levandovskyy2003plural}, see Example \ref{exam_not-right-admissible} and Remark \ref{Rem_notPBW}.\begin{example}\upshape
(i) An Ore extension (also known as a skew polynomial ring, see Section 1.2 of \cite{McConnell-Robson2001}) $R[x;\sigma]$  of automorphism type over an algebra $R$ is a \dd\ algebra over $R$ of type $(1,0)$, while any Ore extension $R[x;\delta]$ of derivation type is a \dd\ algebra of type $(0,1)$.

(ii) Let $0\neq q\in k$ and $I_q$ be the two-sided ideal of the free associative algebra $k\langle x,y\rangle$ generated by the element $yx-qxy$. Then the quotient algebra
$
k_q[x,y]=k\langle x,y\rangle/I_q
$
is called a \emph{quantum plane} (\cite{kassel1995quantum}, Chapter IV).
It is clear that $k_q[x,y]$ is a \dd\ algebra over $k$ of type $(1,1)$.
\end{example}

\

Next we fix some notations.
For $\alpha=(\alpha_1,\dots,\alpha_m)\in \NN^m$, $\beta=(\beta_1,\ldots,\beta_n)\in\NN^n$ and $c\in R$, denote 
$\sigma^{\alpha}(c)=\sigma_1^{\alpha_1}\cdots \sigma_m^{\alpha_m}(c)$,  $|\alpha|=\alpha_1+\cdots+\alpha_m$. {As usual,  $D_i^0 = S_j^ 0 = 1$}, the identity of $R$.
An elements in $\mathcal M=\{\s^{\alpha}\D^{\beta}:\alpha\in\NN^m,\beta\in\NN^n\}$ is called a \emph{standard monomial}.
Moreover, set  $\MM_S=\{\s^{\alpha}:\alpha\in \NN^m\}$
and $\MM_{\D}=\{\D^{\beta}:\beta\in\NN^n\}$.
Let $u=\s^{\alpha}\D^{\beta}\in \MM$. Then the \emph{total degree} of $u$ is defined as $\tdeg(u)=|\alpha|+|\beta|$,
and the \emph{total degree of $u$ with respect to $S_i$} ($D_j$, respectively) is defined as
$\tdeg_{S_i}(u)=\alpha_i$ ($\tdeg_{D_j}(u)=\beta_j$, respectively).
For $i\in \NN$, let $\MM_i=\{u\in\MM:\tdeg (u) =i\}$ and $\MM_{\leq i}=\{u\in\MM:\tdeg (u) \leq i\}$.

Define the \emph{support} of $(0\neq) f\in A$ as
$\Supp(f)=\{u_i\in\MM:f=\sum c_iu_i, 0\neq c_i\in R\}$
and define the \emph{(total) degree} of $f$ as
$\tdeg(f)=\max\{\tdeg(u):u\in\Supp(f)\}$.
Note that if $0\neq c\in R$ then $\Supp(c)=\{\s^{\textbf{\textbf{0}}}\D^{\textbf{0}}\}=\{1\}$ and $\tdeg(c)=0$.
The degree of $0$ is defined as $-\infty$.

With the above notations, we have the following lemmas.
\begin{lemma}\label{lemma_degree}
Suppose $\alpha\in \NN^m$, $\beta\in \NN^n$, and $f,g\in A$.
Then we have
\begin{enumerate}[(i)]
\item $\tdeg(\sigma^{\beta}(\D^{\alpha}))=\tdeg(\D^{\alpha})$.
\item $\tdeg (fg)=\tdeg(f)+\tdeg(g)$.
\end{enumerate}
\end{lemma}
\begin{proof}
One can prove the first statement via the following two steps: $\tdeg(\sigma_i(\D^{\alpha}))=\tdeg(\D^{\alpha}) $ ($1\leq i\leq m$) by induction on $|\alpha|$, and $\tdeg(\sigma^{\beta}(\D^{\alpha}))=\tdeg(\D^{\alpha}) $ by induction on $|\beta|$. The second statement follows from the first one.
\end{proof}

\begin{lemma}\label{lemma_DS=S(D)}
Suppose $\alpha\in \NN^m$ and $ \beta=(\beta_1,\ldots,\beta_n)\in\NN^n$. Then
\begin{enumerate}[(i)]
\item $
\D^{\beta}\s^{\alpha}=\s^{\alpha}\sigma^{\alpha}(\D^{\beta})=
\s^{\alpha}[\sigma_1^{\alpha_1}\sigma_2^{\alpha_2}\cdots \sigma_m^{\alpha_m}(\D^{\beta})].
$
More generally, $f\s^{\alpha}=\s^{\alpha}\sigma^{\alpha}(f)$ for any $f\in R\langle \D\rangle$.
\item $\s^{\alpha}\D^{\beta}=\sigma^{-\alpha}(\D^{\beta})\s^{\alpha}=
[\sigma_1^{-\alpha_1}\sigma_2^{-\alpha_2}\cdots \sigma_m^{-\alpha_m}(\D^{\beta})]\s^{\alpha}$, where $\sigma_i^{-\alpha_i}=(\sigma_i^{-1})^{\alpha_i}$ for $1\leq i\leq m$.
 More generally, $\s^{\alpha}f=\sigma^{-\alpha}(f)\s^{\alpha}$ for any $f\in R\langle \D\rangle$.

\end{enumerate}
\end{lemma}
\begin{proof}
It follows from the definition of \dd\ algebras by induction.
\end{proof}

For $\alpha,\beta\in \NN^n$, we say $\alpha\leq \beta$ if $\alpha_i\leq \beta_i$ for $1\leq i\leq n$.
Define $\alpha+\beta=(\alpha_1+\beta_1,\ldots,\alpha_n+\beta_n)$ and $\alpha-\beta=(\alpha_1-\beta_1,\ldots,\alpha_n-\beta_n)$.

\begin{lemma}\label{lemma_Dr=}
 For any $r\in R$, $1\leq i\leq n$ and $\alpha=(\alpha_1,\ldots,\alpha_i)\in \NN^i$,
$$
\D^{\alpha}r=\sum_{\beta\leq \alpha}\prod_{t=1}^i
{{\alpha_t}\choose{\beta_t}}\delta^{\alpha-\beta}(r)\D^{\beta}
$$
where $\beta=(\beta_1,\ldots,\beta_i)\in\NN^i$.
\end{lemma}
\begin{proof}
It can be proved by induction on $|\alpha|$.
\end{proof}

The following theorem will be used later.

\begin{theorem} [Hilbert Basis Theorem]\label{thm_HibertBT}
If $R$ is Noetherian, then so is the \dd\ algebra $A=R[\s,\D;\sigma,\delta]$. 
\end{theorem}

\begin{proof}
From the point of view of iterated Ore extensions, this theorem is trivial. One can also give a direct proof by using the following fact (cf., \cite{McConnell-Robson2001}, Theorem 1.2.10, and \cite{rowen1988ring}, Proposition 3.5.2):
Let $R$ be a Noetherian ring, and $S$ be an over-ring generated by $R$ and an element $a$ such that $Ra+R=aR+R$. Then $S$ is Noetherian.
\end{proof}

Let us conclude this section by recalling a well-known fact from combinatorics.
For any $q\in \NN$ and $t\in \RR$, denote
$$
{{t}\choose{q}}=\frac{t(t-1)\cdots(t-q+1)}{q!}.
$$
Then, we have the following well-known fact, which will be used latter.
\begin{lemma}\label{lemma_RationalPoly}
Let $q\in \NN$.  Then $h(x)={{x+q}\choose{q}}$ is a polynomial in $x$ of degree $q$, with rational coefficients and positive leading coefficient.
\end{lemma}

\section{\gsbs\  of \dd\ algebras}
Throughout this section, let $R$ be a finite field extension of $k$ and $A=R[\s,\D;\sigma,\delta]$ be a \dd\ algebra of type $(m,n)$, $m,n\in\NN$.
In this section we develop the \gsb\ theory of $A$ and we also show that every left ideal of $A$ has a finite left \gsb.

First we introduce several notations.
For $f\in A$ and a given well-ordering on $\MM$, as in \cite{cox2007ideals}, we use $\lt(f),\lc(f)$ and $ \lm(f)$ to denote the \emph{leading term}, \emph{leading coefficient} and  \emph{leading monomial} of
$f$ respectively. Then we have that $\lt(f)=\lc(f)\cdot \lm(f)$ for any $f\in A$.
Denote $\lm(G)=\{\lm(g):g\in G\}$ for any $G\subseteq A$.

Appropriate orderings on $\MM$ are essential to \gsb\ theory. Mansfield and Szanto \cite{mansfield2003elimination} developed \gsb\ theory for $A$ (with $R=k$) by using the so-called \dd\ ordering defined as follows: Let $>_{\s}$ be a monomial ordering on $\MM_S$ and $>_{\D}$ a total degree monomial ordering on $\MM_{\D}$.
Then the ordering $>$ on the standard monomials $\MM$ defined as follows is called a \emph{differential difference ordering}:
$$
\s^{\alpha}\D^{\beta}> \s^{\alpha'}\D^{\beta'} \Leftrightarrow  \s^{\alpha}>_{\s}\s^{\alpha'}
\mbox{ or } {\alpha}={\alpha'} \mbox{ and } \D^{\beta}>_{\D} \D^{\beta'}.
$$
A differential difference ordering works well for \gsbs\ since it has the following property:
If $>$ is a \dd\ ordering, then $u>v$ for any $u\in\MM_S$ and $v\in \MM_{\D}$.
Now we define a more general class of orderings, which does not necessarily have the above property but still works well for \gsbs.

\begin{definition}\label{Def_DD-ordering}
A \emph{differential difference monomial ordering}, DD-monomial ordering for short, on $\MM$ is a well-ordering $>$ on $\MM$ such that
\begin{center}
if $\s^{\alpha}\D^{\beta}> \s^{\alpha'}\D^{\beta'}$ and $0\neq f\in A$,
then $\lm(f \s^{\alpha}\D^{\beta})> \lm(f\s^{\alpha'}\D^{\beta'})$.
\end{center}
\end{definition}

In other words, a DD-monomial ordering is a {\it left admissible} well-ordering on $\MM$. An ordering $>$ on $\MM$ (or on $\MM_{\s}$, $\MM_{\D}$) is called a {\it monomial ordering} (or {\it admissible ordering}) if it is both left and right admissible; it is called a {\it total degree ordering}  if $u>v$ whenever $\tdeg u>\tdeg v$, $u,v\in \MM$.

Note that, by Proposition 4.1 of \cite{mansfield2003elimination}, any \dd\ ordering is a DD-monomial ordering.
The following example shows that the class of DD-monomial orderings properly includes the class of \dd\ orderings.

\begin{example}\label{Exam_DD-monomial_order}
Suppose $\MM_S$ and $\MM_{\D}$ are well-ordered by {monomial orderings} $>_{\s}$ and $>_{\D}$ respectively. Define an ordering $>$ on $\MM$ as follows:
\begin{eqnarray*}
\s^{\alpha}\D^{\beta}> \s^{\alpha'}\D^{\beta'} \ \Longleftrightarrow&&{\tdeg}(\s^{\alpha}\D^{\beta})>\tdeg(\s^{\alpha'}\D^{\beta'})\\
&\mbox{or}&\tdeg(\s^{\alpha}\D^{\beta})=\tdeg(\s^{\alpha'}\D^{\beta'})\mbox{  and }\s^{\alpha}>_{\s}\s^{\alpha'}\\
&\mbox{or}&\tdeg(\s^{\alpha}\D^{\beta})=\tdeg(\s^{\alpha'}\D^{\beta'}), {\alpha}={\alpha'}\mbox{ and }\D^{\beta}>_{\D} \D^{\beta'}.
\end{eqnarray*}
Then $>$ is a total degree DD-monomial ordering.
\end{example}
\begin{proof} It is clearly by the definition that $>$ is a total degree ordering.
Suppose $\s^{\alpha}\D^{\beta}> \s^{\alpha'}\D^{\beta'}$ and $0\neq f\in A$.
We want to prove that $\lm(f\s^{\alpha}\D^{\beta})> \lm(f\s^{\alpha'}\D^{\beta'})$.
Since $>$ is a total degree ordering, we may suppose that $f$ is homogeneous, i.e., $\tdeg f=\tdeg u$ for any $u\in \Supp(f)$.
Rewrite $f$ as
$$
f=\s^{{\gamma_1}}f_1+\cdots+\s^{\gamma_l}f_l,\ \ \gamma_i\in\NN^m,\ 0\neq f_i\in R\langle \D\rangle, 1\leq i\leq l,\ l\in \NN, \s^{\gamma_1}>_{\s}\cdots>_{\s}\s^{\gamma_l}.
$$

There are three cases.

Case 1: $\tdeg(\s^{\alpha}\D^{\beta})> \tdeg(\s^{\alpha'}\D^{\beta'})$.
Then, by Lemma \ref{lemma_degree},
$$
\tdeg(\lm(f\s^{\alpha}\D^{\beta}))=\tdeg (f)+\tdeg (\s^{\alpha}\D^{\beta}) > \tdeg(f)+\tdeg(\s^{\alpha'}\D^{\beta'})=\tdeg(\lm(f\s^{\alpha'}\D^{\beta'}))
$$ and thus
$ \lm(f\s^{\alpha}\D^{\beta}) >  \lm(f\s^{\alpha'}\D^{\beta'}) $.

Case 2: $\tdeg(\s^{\alpha}\D^{\beta})=\tdeg(\s^{\alpha'}\D^{\beta'})$  and $\s^{\alpha}>_{\s}\s^{\alpha'}$.
By Lemma \ref{lemma_DS=S(D)},
\begin{eqnarray}
f\s^{\alpha}\D^{\beta}&=&\s^{\alpha+\gamma_1}\sigma^{\alpha}(f_1)\D^{\beta}+
\cdots+\s^{\alpha+\gamma_l}\sigma^{\alpha}(f_l)\D^{\beta},\label{eqn_fSD}\\
f\s^{\alpha'}\D^{\beta'}&=&\s^{\alpha'+\gamma_1}\sigma^{\alpha'}(f_1)\D^{\beta'}+
\cdots+\s^{\alpha'+\gamma_l}\sigma^{\alpha'}(f_l)\D^{\beta'}.\nonumber\label{eqn_fS'D'}
\end{eqnarray}
Each $\s^{\alpha+\gamma_i}\sigma^{\alpha}(f_i)\D^{\beta}\neq 0$ and
$\s^{\alpha'+\gamma_i}\sigma^{\alpha'}(f_i)\D^{\beta'}\neq 0$ ($1\leq i\leq l$), since $A$ has no zerodivisors.
Note that $f$ is homogeneous by our assumption, hence, by Lemma \ref{lemma_degree}, we have that
$$
\tdeg(\s^{\alpha+\gamma_1}\sigma^{\alpha}(f_1)\D^{\beta})=\cdots =
\tdeg(\s^{\alpha+\gamma_l}\sigma^{\alpha}(f_l)\D^{\beta}).
$$
But $\s^{\gamma_1}>_{\s}\cdots>_{\s}\s^{\gamma_l}$ and $>_{\s}$ is a monomial ordering, which implies that
$\lm(f\s^{\alpha}\D^{\beta})=\lm(\s^{\alpha+\gamma_1}\sigma^{\alpha}(f_1)\D^{\beta})$.
Similarly,
$\lm(f\s^{\alpha'}\D^{\beta'})=\lm(\s^{\alpha'+\gamma_1}\sigma^{\alpha'}(f_1)\D^{\beta'})$.
Since
$$
\tdeg(\s^{\alpha+\gamma_1}\sigma^{\alpha}(f_1)\D^{\beta})=\tdeg(\s^{\alpha'+\gamma_1}\sigma^{\alpha'}(f_1)\D^{\beta'})\mbox{ and }
\s^{\alpha+\gamma_1}>_{\s}\s^{\alpha'+\gamma_1},
$$
we have that
$$
\lm(f\s^{\alpha}\D^{\beta})=\lm(\s^{\alpha+\gamma_1}\sigma^{\alpha}(f_1)\D^{\beta})>
\lm(\s^{\alpha'+\gamma_1}\sigma^{\alpha'}(f_1)\D^{\beta'})=\lm(f\s^{\alpha'}\D^{\beta'}).
$$

Case 3: $\tdeg(\s^{\alpha}\D^{\beta})=\tdeg(\s^{\alpha'}\D^{\beta'}), {\alpha}={\alpha'}\mbox{ and }\D^{\beta}>_{\D} \D^{\beta'}$.
Then, by Lemma \ref{lemma_DS=S(D)},
\begin{eqnarray}\label{eqn_fSD'}
f\s^{\alpha'}\D^{\beta'}=f\s^{\alpha}\D^{\beta'}=\s^{\alpha+\gamma_1}\sigma^{\alpha}(f_1)\D^{\beta'}+
\cdots+\s^{\alpha+\gamma_l}\sigma^{\alpha}(f_l)\D^{\beta'}.
\end{eqnarray}
From (\ref{eqn_fSD}) and (\ref{eqn_fSD'}),
$$
\lm(f\s^{\alpha}\D^{\beta})=\lm(\s^{\alpha+\gamma_1}\sigma^{\alpha}(f_1)\D^{\beta}),\ \
\lm(f\s^{\alpha'}\D^{\beta'})=\lm(\s^{\alpha+\gamma_1}\sigma^{\alpha}(f_1)\D^{\beta'}).
$$
Since $>_{\D}$ is a monomial ordering on $\MM_D$, $\D^{\beta}>_{\D}\D^{\beta'}$ and $\sigma^{\alpha}(f_1)\in R\langle \D\rangle$ , we have that
$\lm(\sigma^{\alpha}(f_1)\D^{\beta})>_{\D}\lm(\sigma^{\alpha}(f_1)\D^{\beta'})$.
Hence
$$
\lm(f\s^{\alpha}\D^{\beta})=\lm(\s^{\alpha+\gamma_1}\sigma^{\alpha}(f_1)\D^{\beta})>
\lm(\s^{\alpha+\gamma_1}\sigma^{\alpha}(f_1)\D^{\beta'})=\lm(f\s^{\alpha'}\D^{\beta'}).
$$
Therefore, $>$ is a DD-monomial ordering.
\end{proof}

\begin{remark}\upshape
Let $>$ be as in Example \ref{Exam_DD-monomial_order}.
\begin{enumerate}[(i)]
\item The ordering $>$ is not necessarily an extension of $>_{\s}$ or $>_{\D}$. For example, let $>_{\s}$  be a lexicographical orderings on $\MM_S$ with $S_1>_{\s}S_2$. Then $S_1>_{\s}S_2^2$  but $S_1<S_2^2$.
\item The ordering $>$ is not a \dd\ ordering,
because under $>$ a monomial in $\MM_S$ is not necessarily greater than a monomial in $\MM_{\D}$, for example, $S_1<D_1^2$.
\item The ordering $>$  is not right admissible in general, see the following example.
\end{enumerate}
\end{remark}

\begin{example}\label{exam_not-right-admissible}
\upshape
Let $A=R[\s,\D;\sigma,\delta]$ be a \dd\ algebra of type $(1,2)$ with $\sigma_1(D_1)=D_2$ and $ \sigma_1(D_2)=D_1$.
Let $>$ be a DD-monomial ordering as in Example \ref{Exam_DD-monomial_order} with $D_2>_{\D}D_1$. But then
$$
\lm (D_2S_1)=\lm (S_1D_1)=S_1D_1<S_1D_2=\lm(S_1D_2)=\lm(D_1S_1).
$$
Thus $>$ is not right admissible.
\end{example}

\begin{remark}\label{Rem_notPBW}\upshape
The above example (where $D_2S_1=S_1D_1$) also shows that a \dd \ algebra is not necessarily an algebra of solvable type \cite{Weispfenning1990non} (or a PBW extension \cite{bell1988uniform}, or a G-algebra \cite{levandovskyy2003plural}).
\end{remark}

Let $f,g\in A$.
If there exists $h\in A$ such that
$f=hg$, we say that $f$ is {\emph{right divisible}}
 by $g$, and $g$ ($h$, respectively) is called a \emph{right factor} or \emph{right quotient} (\emph{left factor} or \emph{left quotient}, respectively) of $f$.
Denote the left quotient $h=\LQ{f}{g}$.
With the above definitions and notations, we have the following lemma.
\begin{lemma}\label{lemma_f/g}
Suppose $f,g,h\in A$.
\begin{enumerate}[(i)]
\item If $g$ is a right factor of $f$ then $\LQ fg=\LQ{fh}{gh}$,
but $\LQ fg\neq \LQ{hf}{hg}$ in general.
\item  If $hg$ is a right factor of $f$ then $\LQ f{hg}\cdot h=\LQ{f}{g}$ and $h\cdot \LQ f{hg}= \LQ{hf}{hg}$,
but $h\cdot \LQ f{hg}\neq \LQ{f}{g}$ in general.
\item If $g$ is a right factor of $f$ and $h$ is a right factor of $g$, then
$\LQ fg \cdot \LQ gh=\LQ fh$.
\end{enumerate}
\end{lemma}

\begin{proof}
It is obvious by definition.
\end{proof}
\begin{lemma}\label{lemma_RightDivisible-iff}
$\s^{\alpha}\D^{\beta}$ is right divisible by $\s^{\alpha'}\D^{\beta'}$ \IFF
$$
\alpha=\alpha'+\gamma \mbox{ and } \beta=\beta'+\gamma'\ \mbox{ for some }\gamma\in\NN^m \mbox{ and } \gamma'\in\NN^n.
$$
\end{lemma}
\begin{proof}
The ($\Rightarrow$) part is clear.

($\Leftarrow$). Suppose $
\alpha=\alpha'+\gamma \mbox{ and } \beta=\beta'+\gamma'\ \mbox{ for some }\gamma\in\NN^m \mbox{ and } \gamma'\in\NN^n
$. Then, by Lemma \ref{lemma_DS=S(D)},
\begin{eqnarray*}
\s^{\alpha}\D^{\beta}=\s^{\alpha'+\gamma}D^{\beta'+\gamma'}=\s^{\gamma}\sigma^{-\alpha'}(D^{\gamma'})\cdot \s^{\alpha'}\D^{\beta'}.
\end{eqnarray*}
Hence, $\s^{\alpha}\D^{\beta}$ is right divisible by $\s^{\alpha'}\D^{\beta'}$.
\end{proof}

However, the ``only if'' part of the above  lemma is  {not} true for the left division.
In fact, in Example \ref{exam_not-right-admissible}, $S_1D_2=D_1S_1$ and thus $S_1D_2$ is left divisible by $D_1$, but the exponents of them do not have the relation stated in Lemma \ref{lemma_RightDivisible-iff}.

\

From now on to the end of this section, we fix a DD-monomial ordering $>$ on $\MM$.
Then we have the following lemma, which is similar to Corollary 4.3 of \cite{mansfield2003elimination}.

\begin{lemma}\label{lemma_leadingTermProduct}
For any $f,g\in A$, we have
\begin{enumerate}[(i)]
\item $\lm(fg)=\lm(f\cdot\lm(g))=h\cdot\lm(g), \mbox{ for some }h\in A$.
\item $
\lt(fg)=\lt(f\cdot\lt(g))=h'\cdot\lt(g), \mbox{ for some }h'\in A.
$
\end{enumerate}
Furthermore, $h$ and $h'$ in the above are uniquely determined by $f$ and $g$.
\end{lemma}

\begin{definition}\label{Def_GB}
Let $I$ be a left ideal of $A$.
A  set $G\subseteq I$ is called a \emph{left \gsb}  of $I$ with respect to $>$ if,
for any $0\neq f\in I$, there exists $g\in G$ such that $\lm(f)$ is right divisible by $\lm(g)$.
\end{definition}

Note that we do not require a \gsb\ to be finite.

Let $G\subseteq A$. Define the irreducible words with respect to $G$ as
$$\Irr(G)=\{w\in \MM:w\not=f\lm(g) \mbox{ for any }g\in G, f\in A\}.
$$
By Lemma \ref{lemma_leadingTermProduct}, it is easy to see that
\begin{eqnarray*}
\Irr(G)&=&\{w\in \MM:w\not=\lm(f\lm(g)) \mbox{ for any }g\in G, f\in A\}\\
&=&\{w\in \MM:w\not=\lm(u\lm(g)) \mbox{ for any }g\in G, u\in \MM\}.
\end{eqnarray*}

Let $f,h, g\in A$ and $G\subseteq A$.
Then \emph{$f$ reduces to $h$ modulo $g$},
denoted by $f\rightarrow_g h$, if $h=f-qg$ and $\lt(f)=\lt(qg)$ for some $q\in A$.
We say that \emph{$f$ reduces to $h$ modulo $G$},
denoted by  $f\rightarrow_G {h}$, if there exists a finite chain of reductions
$$
f\rightarrow_{g_1} f_1\rightarrow_{g_2} f_2\rightarrow_{g_3}\cdots \rightarrow_{g_t} f_t=h,
$$
where each $g_i\in G$ and $t\in \NN$.
{Furthermore}, if {${\Supp(h)}\subseteq \Irr(G)$}, them $h$ is  \emph{irreducible} with respect to $G$,
and we call $h$ a \emph{remainder} of $f$ modulo $G$.

With these definitions, we have the following lemma.
\begin{lemma}\label{lemma_remainderUnique}
Let $G\subseteq A$ be a finite set and $f\in A$. Then,
\begin{enumerate}[(i)]
\item
$
f=\sum  c_iu_i g_i+r,
$
where each $c_i\in R, u_i\in \MM, g_i\in G$, $\lm(u_i g_i)\leq \lm(f)$ and $r$ is a remainder of $f$ modulo $G$.
\item furthermore, if $G$ is a left \gsb\   for a left ideal  of $A$, then the {remainder of $f$ modulo $G$} is unique
(denoted by $\Rem(f,G)$).
\end{enumerate}
\end{lemma}

\begin{proof}
(i) It can be proved by induction on $\lm(f)$.

(ii)
In order for a contradiction, we suppose that both $r$ and $r'$ are reminders of $f$ modulo $G$ and $r\neq r$.
Then $0\neq r-r'=(f-r')-(f-r)\in I$.
Hence $\lm(r-r')\not \in \Irr(G)$ by the definition of \gsbs.
But $\lm(r-r')\in \Supp(r)\cup\Supp(r')\subseteq \Irr(G)$, a contradiction.
\end{proof}

Note that, in general,  a remainder of $f\in A$ modulo some subset $G'\subseteq A$ is not unique.

\begin{lemma}\label{lemma_representation_of_f}\upshape
Let $f\in A$ and $G\subseteq A$. Then $f$ can be written as
$$
f=\sum\limits_{i=1}^{s} a_iu_i+\sum\limits_{j=1}^{t} b_jv_jg_j,
$$
where $s,t\in \NN$,  each $a_i,b_j\in R, u_i\in \Irr(G), v_j\in \MM, g_i\in G$ and
$$
\lm(f)\geq\lm(u_1)>\cdots >\lm(u_s),\ \ \lm(f)\geq\lm(v_1g_1)>\cdots >\lm(v_tg_t),
$$
exactly one  of those  $\lm(u_i)$ and $\lm(v_jg_j)$ is equal to $\lm(f)$.
\end{lemma}
\begin{proof} (By induction on $\lm(f)$.)
If $\lm(f)=1$, then $f\in R$ and the statement holds clearly.
Suppose that the statement holds for any polynomial with leading monomial less than $\lm(f).$
We need to show that it also holds for $f$.
Define $f_1$ as follows.
If $\lm(f)\in \Irr(G)$, then set $f_1=f-\lc(f)\lm(f)$.
If  $\lm(f)\not\in \Irr(G)$, {i.e.,} there exist $g\in G$ and $v\in \MM$ such that
$\lm(f)=\lm(vg)$, then set $f_1=f-bvg$ where $b=\lc(f)\lc(vg)^{-1}\in R$.
Then $\lm(f_1)< \lm(f)$ in either case. Hence, by induction hypothesis, $f_1=\sum_{1\leq i\leq s_1}a_iu_i+\sum_{1\leq j\leq t_1} b_jv_jg_j$
where $s_1,t_1\in \NN$, each $a_i,b_j\in R, g_i\in G, u_i\in \Irr(G), v_j\in \MM$,
$
\lm(f_1)\geq\lm(u_1)>\cdots >\lm(u_{s_1})$ and $ \lm(f_1)\geq\lm(v_1g_1)>\cdots >\lm(v_tg_{t_1})$.
Thus $f=f_1+\lc(f)\lm(f)$ (or  $f=f_1+ bvg$) has the desired presentation.
\end{proof}

The following theorem solves the ideal membership problem of $A$.
\begin{theorem}\label{thm_membership}
Let $G\subseteq A$ be a left \gsb\ for a left ideal $I$, and let $f\in A$. Then $f\in I$ if and only if $\Rem(f,G)=0$.
\end{theorem}
\begin{proof}
If $f\in I$ then, by Lemma \ref{lemma_remainderUnique},  $\Rem(f,G)=0$. On the other hand, if $\Rem(f,G)=0$, then $f=f-\Rem(f,G)\in I$.
\end{proof}

The following proposition indicates that the definition of \gsbs\ in this paper is {equivalent} to Definition 4.5 of \cite{mansfield2003elimination} if the ordering under consideration is a \dd\ ordering.

\begin{proposition}\label{prop_GB_iff_<lm(G)>=<lm(I)>}
Let $I$ be a left ideal of $A$ and $G\subseteq I$. Then

(i) $G$ is  a \gsb\ for $I$ \IFF $\lm(G)$ and $\lm(I)$ generate the same left ideal of $A$.

(ii) If $G$ is a left \gsb\ for $I$, then $G$ generates $I$ as a left ideal of $A$.
\end{proposition}

\begin{proof}
(i) ($\Rightarrow$) It follows from the definition of \gsbs.

($\Leftarrow$) Let $0\neq f\in I$. Since $\lm(G)$ and $\lm(I)$ generate the same left ideal of $A$, $\lm(f)=\sum a_i\lm(g_i)$ where $a_i\in A, g_i\in G$. Then $\lm(f)\in \Supp(a_i\lm(g_i))$ for some $i$.
Hence $\lm(f)$ is right divisible by $\lm(g_i)$.

(ii) Suppose $f\in I$. By Lemma \ref{lemma_remainderUnique}, we can write $f=\sum b_ig_i+r$ where $b_i\in A$, $g_i\in G$ and $r$ is irreducible with respect to  $G$.
Thus $r=f-\sum b_ig_i\in I$.
Suppose $r\neq 0$. Since $G$ is a \gsb\ for $I$, $\lm(r)$ is right divisible by $\lm(g_i)$ for some $g_i\in G$,
contradicting our assumption that $r$ is irreducible with respect to  $G$.
Hence $r=0$ and thus $f=\sum b_ig_i$ is in the left ideal of $A$ generated by $G$.
Therefore, $I$ is generated by $G$ as a left ideal of $A$.
\end{proof}

\

Let $\s^{\alpha}\D^{\beta},\s^{\alpha'}\D^{\beta'}\in \MM$. Then
the \emph{least common left multiple} of $\s^{\alpha}\D^{\beta}$ and $\s^{\alpha'}\D^{\beta'}$ is defined as
$$
\lclm(\s^{\alpha}\D^{\beta},\s^{\alpha'}\D^{\beta'})=\s^{\mu}\D^{\nu},
$$
where $\mu=(\mu_1,.\ldots,\mu_m)\in \NN^m$, $\nu=(\nu_1,.\ldots,\nu_n)\in \NN^n$, each $\mu_i=\max\{\alpha_i,\alpha'_i\}$,
$\nu_j=\max\{\beta_j,\beta'_j\}$.
For the sake of convenience, for any $f,g\in A$,
$\lclm(\lm(f),\lm(g))$ is sometimes denoted by {$\lclm(f,g)$}.
Let $f,g\in A$. Then there exists a unique pair of polynomials $f',g'\in A$ such that
$$
f'\lt(f)=\lclm(f,g)=g'\lt(g).
$$
Then, the polynomial $f'f-g'g\in A$ is called {the} \emph{S-polynomial} of $f$ and $g$, denoted by
$\Spoly(f,g)$, that is,
\begin{eqnarray}
\Spoly(f,g)&=&f'f-g'g\nonumber\\
&=&\LQ{\lclm(f,g)}{\lt(f)}f-\LQ{\lclm(f,g)}{\lt(g)}g.\label{eqn_Spoly}
\end{eqnarray}
Note that $\lc(f'\lt(f))=1$.

\begin{lemma}\label{lemma_Cancle_lm}
Let $f,g\in A$. Then
$
\lm(\Spoly(f,g))<\lclm(f,g).
$
\end{lemma}

\begin{proof} By Lemma \ref{lemma_leadingTermProduct},
$\lt(f'f)=\lt(f'\lt(f))=\lt(g'\lt(g))=\lt(g'g).$ Hence
$
\lm(\Spoly(f,g))=\lm(f'f-g'g)<\lm(f'f)=\lclm(f,g).
$
\end{proof}

The  following lemma can be proved by using a telescoping argument as in Lemma 5 of Chapter 2.6 in  \cite{cox2007ideals}.
\begin{lemma}\label{lemma_key2}(cf. \cite{mansfield2003elimination}, Lemma 4.8, and \cite{cox2007ideals}, Lemma 5 of Chapter 2.6)
Let $f=c_1f_1+\cdots+c_sf_s$ for $f_1,\ldots,f_s\in A$, $c_1,\ldots,c_s\in R$ and $s\in\NN$. Suppose that for all $1\leq i\leq s$, $\lm(f_i)=u $ for some $u\in\MM$.
If $\lm(f)<u$, then there exist $d_{ij}\in R$ $(1\leq i,j\leq s)$ such that
$$
f=\sum_{i,j=1}^sd_{ij}\Spoly(f_i,f_j).
$$
\end{lemma}
\begin{theorem}\label{thm_s-poly=GB}
Let $G\subseteq A$ and $I$ be the left ideal of $A$ generated by $G$. Then  $G$ is a left \gsb\ for $I$ if and only if $\Spoly(g_1,g_2)\rightarrow_G 0$ for any $g_1,g_2\in G$.
\end{theorem}

\begin{proof}
($\Rightarrow$) Suppose $G$ is a \gsb\ for $I$. Since $\Spoly(g_1,g_2)\in I$ for any $g_1,g_2\in G$, by Theorem \ref{thm_membership}, $\Rem(\Spoly(g_1,g_2),G)=0$, i.e., $\Spoly(g_1,g_2)\rightarrow_G0$ as desired.

($\Leftarrow$)
Suppose that $\Spoly(g_1,g_2)\rightarrow_G 0$ for any $g_1,g_2\in G$.
We want to prove that for any $0\neq f\in I$, $\lm(f)$ is right divisible by $\lm(g)$ for some $g\in G$.
Since $f\in I$, $f$ can be written as
\begin{eqnarray}\label{eqn_expression_f}
f=\sum_{i=1}^th_ig_i,\ \ t\in\NN,\ h_i\in A,\ g_i\in G, \ \lm(h_1g_1)\geq \lm(h_2g_2)\geq \cdots\geq \lm(h_tg_t).
\end{eqnarray}
Let $u=\lm(h_1g_1)$. Assume that among all possible expression of $f$ of the form (\ref{eqn_expression_f}) we chose one with minimal $u$. Suppose
\[
\lm(h_1g_1)=\lm(h_2g_2)=\cdots=\lm(h_sg_s)>\lm(h_{s+1}g_{s+1})\geq\cdots\geq \lm(h_tg_t), \ \ s\in \NN.
\]

Note that $\lm(f)\leq u$. Now we prove that $\lm(f)=u$. In order for a contradiction, we suppose that $\lm(f)<u$.
By Lemma \ref{lemma_leadingTermProduct},
$\lt(h_ig_i)=h_i'\lt(g_i)$ for some $h_i'\in A$, $1\leq i\leq t$.
Then
$\lt(h_ig_i)=\lt(h'_i\lt(g_i))=\lt(h_i'g_i)$ and thus $\lm((h_i-h_i')g_i)<\lm(h_ig_i)$.
Rewrite
\begin{eqnarray}\label{eqn_f=++}
f=\sum_{\lm(h_ig_i)=u}h'_ig_i+\sum_{\lm(h_ig_i)=u}(h_i-h_i')g_i+\sum_{\lm(h_lg_l)<u}h_lg_l.
\end{eqnarray}
Then
$$\lm\left(\sum_{\lm(h_ig_i)=u}(h_i-h_i')g_i+\sum_{\lm(h_lg_l)<u}h_lg_l\right)<u.$$
Hence $\displaystyle\lm\left(\sum_{\lm(h_ig_i)=u}h'_ig_i\right)<u$ since $\lm(f)<u$.
By Lemma \ref{lemma_key2}, there exist $d_{ij}\in k$ for $1\leq i,j\leq s$ such that
$$
\sum_{\lm(h_ig_i)=u}h'_ig_i=\sum_{i,j=1}^sd_{ij}\Spoly(h'_ig_i,h'_jg_j).
$$
Since $\lclm(\lm(h_i'g_i),\lm(h_j'g_j))=\lclm(u,u)=u$ for all $1\leq i,j\leq s$,
by  Equation (\ref{eqn_Spoly}) and Lemma \ref{lemma_f/g},
\begin{eqnarray*}
\Spoly(h'_ig_i,h'_jg_j)&=&\LQ{u}{\lt(h_i'g_i)}h_i'g_i-\LQ{u}{\lt(h_j'g_j)}h_j'g_j\\
&=&\LQ{u}{h_i'\lt(g_i)}h_i'g_i-\LQ{u}{h_j'\lt(g_j)}h_j'g_j\\
&=&\LQ{u}{\lt(g_i)}g_i-\LQ{u}{\lt(g_j)}g_j\\
&=&\LQ{u}{\lclm(g_i,g_j)}\LQ{\lclm(g_i,g_j)}{\lt(g_i)}g_i\\
&&-\LQ{u}{\lclm(g_i,g_j)}\LQ{\lclm(g_i,g_j)}{\lt(g_j)}g_j\\
&=&\LQ{u}{\lclm(g_i,g_j)}\Spoly(g_i,g_j).
\end{eqnarray*}
Hence
$$
\sum_{\lm(h_ig_i)=u}h'_ig_i=\sum_{i,j=1}^sd_{ij}\LQ{u}{\lclm(g_i,g_j)}\Spoly(g_i,g_j)\longrightarrow_G0.
$$
Thus by Lemma \ref{lemma_remainderUnique} (i),
$$
\sum_{\lm(h_ig_i)=u}h'_ig_i=\sum_{j\in J}c_ju_jg_j,\ \  \lm(u_jg_j)\leq \lm\left(\sum_{\lm(h_ig_i)=u}h'_ig_i\right)<u,
$$
where $c_j\in k, u_j\in \MM, g_j\in G$ and $J$ is a finite index set.
Now we can rewrite (\ref{eqn_f=++}) as
\begin{eqnarray}\label{eqn_f=++'}
f=\sum_{j\in J}c_ju_jg_j+\sum_{\lm(h_ig_i)=u}(h_i-h_i')g_i+\sum_{\lm(h_lg_l)<u}h_lg_l,
\end{eqnarray}
where $\lm(u_jg_j)<u$, $\lm((h_i-h_i')g_i)<u$ and $\lm(u_lg_l)<u$.
That is, (\ref{eqn_f=++'}) is an expression of $f$ of form (\ref{eqn_expression_f}) with all leading monomials of summands less than $u$,
which contradicts the minimality of $u$. Therefore, we proved that $\lm(f)=u$.

Now, by Lemma \ref{lemma_leadingTermProduct},
$\lm(f)=u=\lm(h_1g_1)=h'\lm(g_1)$ for some $h'\in A$.
That is, $\lm(f)$ is right divisible by $\lm(g_1)$.
\end{proof}

Now we have the following algorithm.
\begin{algorithm}[Left \gs\ Basis Algorithm]\label{Alg_GB}
\upshape\

Input: $F=\{f_1,\ldots,f_s\}\subset A$, $s\in\NN$, each $f_i\neq 0$.

Output: a \gsb\ $G=\{g_1,\ldots,g_t\}$ for $I$, $t\in \NN$, with $F\subseteq G$.

Initialization: $G:=F$, $P:=\{\{p,q\}:p\neq q, p,q\in F\}$.

WHILE $P\neq\emptyset$ DO

\hspace{12mm} Choose any pair $\{p,q\}\in P$, $P:=P\setminus \{\{p,q\}\}$

\hspace{12mm} $r:=$ a remainder of $\Spoly(p,q)$ modulo $G$

\hspace{12mm} IF $r\neq 0$ THEN

\hspace{12mm} $G:=G\cup \{r\}$, $P:=P\cup \{\{g,r\}: g\in G\}$

END DO
\end{algorithm}
\begin{theorem}\label{Thm_GBAlgorithm}
Let $I$ be a left ideal of $A$ generated by nonzero elements $f_1,\ldots, f_l\in A$. Then Algorithm \ref{Alg_GB} returns a finite \gsb\ for $I$.
\end{theorem}

\begin{proof}
 We first prove that the algorithm terminates after finitely many steps. After each pass through the main loop (i.e., the while loop), if there is a nonzero remainder $r$, then $G$ consists of the old $G$ (denoted by $G'$) together with the nonzero remainder $r$, i.e., $G=G'\cup\{r\}$. Since $r\in \Irr(G')$, it is easy to show that $\lt(G')A$, the left ideal of $A$ generated by $\lt(G')$, is properly contained in $\lt(G)A$. By the Hilbert Basis Theorem \ref{thm_HibertBT}, any ascending chain of left ideals of $A$ will stabilize. Hence, after finitely many iterations of the main loop, there is no nonzero remainder any more. Hence $P$ will be exhausted (i.e., $P=\emptyset$) after finitely many iterations of the main loop, and then the algorithm terminates and return $G$.

By Theorem \ref{thm_s-poly=GB}, the return $G$ is a \gsb\ for $I$.
\end{proof}

\begin{theorem}\label{Thm_GB_existence}
Let $A=R[\s,\D;\sigma,\delta]$ be a \dd\ algebra, where $R$ is a field extension of $k$.
Then every left ideal of $A$ has a finite left \gsb.
\end{theorem}
\begin{proof}
It follows by Theorem \ref{Thm_GBAlgorithm} and the Hilbert Basis Theorem \ref{thm_HibertBT}.
\end{proof}

The following theorem plays a key role in the Gelfand-Kirillov dimension computation in the next section.
\begin{theorem}\label{Thm_GB-linear+basis}
Let $A=R[\s,\D;\sigma,\delta]$ be a \dd\ algebra, where $R$ is a field extension of $k$.
Let $G\subseteq A$, and $I$ be the left ideal of $A$ generated by $G$. For $f\in A$, let $\overline{f}=f+I$, which belongs to the left $A$-modulo $A/I$.
Then $G$ is a left \gsb\ for $I$ if and only if the set $B=\{\overline{u}|u\in\Irr(G)\}$ is an $R$-basis of the left $A$-module $A/I$.
\end{theorem}
\begin{proof}
($\Rightarrow$) Suppose that $G$ is a left \gsb\  for $I$. For any $\overline{f}=f+I\in A/I$, by Lemma \ref{lemma_representation_of_f},
\begin{eqnarray*}
\overline{f}=\sum\limits_{i=1}^{s} a_iu_i+\sum\limits_{j=1}^{t} b_jv_jg_j+I=\sum\limits_{i=1}^{s} a_iu_i+I
            =\sum\limits_{i=1}^{s} a_i(u_i+I)
            =\sum\limits_{i=1}^{s} a_i\overline{u_i},
\end{eqnarray*}
where $s,t\in \NN$,  each $a_i,b_j\in R, u_i\in \Irr(G), v_j\in \MM, g_i\in G$.
Thus, $B$ spans $A/I$ as an $R$-space.

In order to prove that $B$ is $R$-linearly independent,
suppose that
$$
a_1\overline{u_1}+a_2\overline{u_2}+\cdots+a_t\overline{u_t}=\overline{0},\  a_i\in R, \ u_i\in \Irr(G),\ i=1,\ldots,t, \ \ t\in \NN,
$$
and that $u_1{>} u_2>\cdots> u_t$. Let $f=a_1 {u_1}+a_2 {u_2}+\cdots+a_t {u_t}$. Then $\overline f=\overline0$ and thus $f\in I$.
If $a_1\not=0$ then $\lm(f)=u_1\in \Irr(G)$; but $\lm(f)\not\in \Irr(G)$ since $G$ is a left \gsb\ for $I$.
Hence $a_1=0$ and similarly $a_2=a_3=\cdots=a_t=0$. Therefore $B$ is $R$-linearly independent.

($\Leftarrow$) Suppose $B$ is an $R$-basis of  $A/I$ and $0\not= f \in I$.
It suffices to prove that $\lm(f)$ is right divisible by $\lm(g)$ for some $g\in G$. By Lemma \ref{lemma_representation_of_f}, we can write
$
f=\sum\limits_{i=1}^{s} a_iu_i+\sum\limits_{j=1}^{t} b_jv_jg_j,
$
where $s,t\in \NN$, each $a_i,b_j\in R, u_i\in \Irr(G), v_j\in \MM, g_i\in G$ and
$$
\lm(f)\geq\lm(u_1)>\cdots >\lm(u_s),\ \ \lm(f)\geq\lm(v_1g_1)>\cdots >\lm(v_tg_t).
$$
 Then $\overline0=\overline f=a_1\overline{u_1}+\cdots+a_s\overline{u_s}$.
Since $B$ is an $R$-basis of $A/I$, $a_1=\cdots=a_s=0$. Hence $\lm(f)=\lm(v_1g_1)$ and thus, by Lemma \ref{lemma_leadingTermProduct}, $\lm(f)$ is right divisible by $\lm(g_1)$.
\end{proof}

\section{Gelfand-Kirillov dimension of cyclic $A$-modules}\label{Sec_GK_A/I}
In this section, we compute the Gelfand-Kirillov dimension of cyclic modules over a \dd\ algebra.
We assume that the reader is familiar with the notions of gradings and filtrations of algebras and modules. We refer the reader to  Chapter 6 of \cite{Krause-Lenagan_2000} for more details.

Let $A$ be a graded $k$-algebra  and let $M= \bigoplus_{i\in\NN}M_i$ be a graded left $A$-module.
Then the \emph{Hilbert function} of $M$ is defined as the mapping:
$$
\HF_M:\mathbb{N}\to\mathbb{N},\ \HF_M(i)=\dim_k(M_0\oplus \cdots \oplus M_i),\ i\in \mathbb{N}.
$$
The following lemma relates the Gelfand-Kirillov dimension and the Hilbert function of a \fg\ module over a \fg\ algebra.

\begin{lemma}\label{lemma_Hilbet_GK} (\cite{Krause-Lenagan_2000}, Lemma 6.1)
If $A$ is a \fg\ $k$-algebra and $M$ is a \fg\ left $A$-module, then
$\GK(M)=\overline{\lim\limits_{i\to \infty}}\log_i \HF_M(i)$.
\end{lemma}

From now on to the end of this section, we fix the following notations. Let $A=R[\s,\D;\sigma,\delta]$ be a \dd\ algebra of type $(m,n)$, { $m,n\in \NN$}, where $R$ is a finite field extension of $k$ (it is easy to see that if $\dim_kR=\infty$ then $\GK(M)=\infty$). Let $V=\{v_1,\ldots,v_d\}$ ($d\in \NN$) be a $k$-basis of $R$ and $I$ be a proper left ideal of $A$.

Denote the  left $A$-module $A/I$ by $M$.
Let $A_i$, $i\in\NN$, be the $R$-subspace of $A$ spanned by $\MM_i=\{\s^{\alpha}\D^{\beta}:|\alpha|+|\beta|=i\}$.
Then each $A_i$ is spanned as a $k$-space by $\{vu:v\in V, u\in \MM_i\}$ and $\{A_i\}_{i\in \NN}$ is a grading of $A$, i.e.,
$$
A=\bigoplus_{i\in \NN}A_i, \ \ \mbox{and } A_i\cdot A_j\subseteq A_{i+j}\ \mbox{ for all }i, j\in\NN.
$$
It induces a grading of the left $A$-module $M=A/I=\bigoplus_{i\in\NN} M_i$, where
$M_i= (A_i+I)/I$ is a $k$-subspace of $M$.
Recall that if $f\in A$ we denote the element $f+I$ of $M$ by $\overline f$.
\begin{proposition}\label{prop_Hilbert_function}
Let $G\subseteq A$ be a left \gsb\ for $I$ with respect to a {total degree DD-monomial ordering}. Then the following hold:

(i) The set $B_i=\{\overline{vu}:v\in V,\ u\in\Irr(G)\cap\MM_{\leq i}\}$ is a $k$-basis of $M_{\leq i}=M_0\oplus\cdots\oplus M_i$,  $i\in\NN$.

(ii) The Hilbert function of  $M$ is given by
$$
\HF_{M}(i)=d\cdot |\Irr(G)\cap \MM_{\leq i} |, \ i\in \mathbb{N}.
$$
\end{proposition}
\begin{proof}
(i) First we prove that $B_i'=\{\overline{u}:  u\in\Irr(G)\cap\MM_{\leq i}\}$ is an $R$-base of $M_{\leq i}$ as a left $R$-module. By Theorem \ref{Thm_GB-linear+basis}, $B'_i$ is linearly independent over $R$ for any $i\in \NN$. Thus it suffices to prove that $B_i'$ spans $M_{\leq i}$ as a left $R$-module.
Note that $M_{\leq i}=\{\overline f\in M:f\in A, \tdeg f\leq i\}$. For $f\in A$ with $\tdeg f\leq i$, by Lemma \ref{lemma_representation_of_f},
$
\overline f=\sum\limits_{j=1}^{s} a_j\overline{u_j},
$
where $s\in \NN$,  each $a_j \in R, u_j\in \Irr(G)$ and
$
\lm(f)\geq\lm(u_1)>\cdots >\lm(u_s).
$
Since $>$ is a total degree DD-monomial ordering, for $1\leq j\leq s$, $\tdeg u_j\leq \tdeg f\leq i$ and thus $\overline{u_j}\in B'_i$.
Hence $B'_i$ spans $M_i$.

Since $V$ is a $k$-basis of $R$, we have that $B_i=\{\overline{vu}:v\in V,\ u\in\Irr(G)\cap\MM_{\leq i}\}$ is a $k$-basis of $M_{\leq i}$.

(ii) By (i), it is sufficient to show that $|B'_i|=|\Irr(G)\cap \MM_{\leq i} |$.
It is clear that $|B'_i|\leq |\Irr(G)\cap \MM_{\leq i}|$.
For the other direction, for $u,v\in \Irr(G)\cap \MM_{\leq i}$ with $\overline u=\overline v$, we want to prove that $u=v$. If $u\neq v$, without loss of generality, we suppose $u>v$.
Then $0\neq f=u-v\in I$ and thus $u=\lm f\not\in \Irr(G)$ since $G$ is a \gsb\ for $I$, a contradiction.
Hence $u=v$. Thus $|B_i'|= |\Irr(G)\cap \MM_{\leq i}|$.
\end{proof}

For convenience, denote $x_i=S_i, x_{m+j}=D_j$ for $1\leq i\leq m, 1\leq j\leq n$ and let $l=m+n$.
Denote $X^{\alpha}=x_1^{\alpha_1}x_2^{\alpha_2}\cdots x_l^{\alpha_l}$ for $\alpha=(\alpha_1,\ldots,\alpha_l)\in\NN^l$.
Then ${\MM=\{X^{\alpha}:\alpha\in\NN^{l}\}}$.
For $u=X^{\alpha}\in\MM$ and $p\in\mathbb{N}$, we define that (cf., Section 9.3 of \cite{Becker-Weispfenning_1993})
$$
\topp_p(u)=\{i\in\NN:1\leq i\leq l, \alpha_i\geq p\}
$$
and
$$
\sh_p(u)=X^{\beta}, \ \mbox{where } \beta=(\beta_1,\ldots,\beta_l)\in\NN^l,\beta_i=\min\{p,\alpha_i\}, 1\leq i\leq l,
$$
i.e., $\topp_p(u)$ is the set of indices where ``$u$ tops $p$'' and $\sh_p(u)$ is ``$t$ shaved at $p$''.
For $W\subseteq \MM$, we define a relation $\sim_p$ on $W$  as follows: for any $u,v\in W$,
$$
u\sim_p v \mbox{ if } \sh_p(u)=\sh_p(v).
$$
Then we have the following lemma, whose proof is straightforward.
\begin{lemma}\label{lemma_equi_relation_on_W}
Let $W\subseteq \MM$ and $p\in\NN$. Then

(i) $\sim_p$ is an equivalence relation on $W$.

(ii) Let $[u]_p=\{v\in W:u\sim_pv\}$ and $W/_{\sim_p}=\{[u]_p: u\in W\}$. If $\sh_p(u)\in W$ for all $u\in W$, then the set
$$
W_p=\{u\in W:\sh_p(u)=u\}
$$
 is a set of normal forms for $W/_{\sim_p}$ (i.e., a system of unique representatives for $W/_{\sim_p}$). The set $W_p$ can also be described as
$$W_p=\{X^{\alpha}\in W: \alpha_i\leq p, \ 1\leq i\leq l\}.$$

(iii) For any $u=X^{\alpha}\in W_p$,
$$
[u]_p=\{X^{\beta}\in W: \alpha\leq \beta, \ \alpha_i=\beta_i \mbox{ for } i\not\in \topp_p(u)\}.
$$
\end{lemma}

Now we are in a position to prove our main theorem in this section.
\begin{theorem}\label{thm_GK=degHilbertPoly}
Let $R$ be a finite field extension of $k$. Let $A=R[\s,\D;\sigma,\delta]$ be a \dd\ algebra, $I$ be a left ideal of $A$ and $G$ be a finite \gsb\ for $I$ with respect to a total degree DD-monomial ordering.
Set $p=\max\{{\tdeg_{x_i}}(\lm(g)):g\in G, 1\leq i\leq l\}.$
Then

(i) There exists a unique polynomial $h\in\QQ[x]$ such that the Hilbert function $\HF_{M}$ of the left $A$-module $M=A/I$ satisfies
 $\HF_{M}(t)=h(t) \mbox{ for all } t\geq lp.$

(ii) The Gelfand-Kirillov dimension of  $M$ is equal to the degree of $h$, which is given as,
$$
\GK(M)=\deg h=\max\{|\topp_p(u)|:u\in \Irr(G)\cap \MM_{\leq t}, \sh_p(u)=u\}\mbox{ for any }t\geq lp.
$$
\end{theorem}
\begin{proof}(i)
(Existence)
Let $t\in\NN$ and $t\geq lp$.
We will construct the desired polynomial $h$ by counting the elements of the set {$W=\Irr(G)\cap \MM_{\leq t}$}.
By Proposition \ref{prop_Hilbert_function}, $
\HF_{A/I}(t)=d\cdot |W|.
$
Note that if $u\in W$ then $\sh_p(u)\in W$.
By Lemma \ref{lemma_equi_relation_on_W} (ii),
$
W_p=\{u\in W:\sh_p(u)=u\}
$
is a set of normal forms of $W/_{\sim_p}$ and thus
\begin{eqnarray}\label{eqn_CardW}
|W|=\sum_{u\in W_p}|[u]_{ p }|,
\end{eqnarray}
where $[u]_{p}$ is the equivalence class of $u$ with respect to $\sim_p$.
By Lemma \ref{lemma_equi_relation_on_W} (iii),
supposing $u=X^{\alpha}\in W_p$,
$$
[u]_p=\{X^{\beta}\in W: \alpha\leq \beta, \ \alpha_i=\beta_i \mbox{ for } i\not\in \topp_p(u)\}.
$$
Hence,
\begin{eqnarray}\label{eqn_Card[u]}
|[u]_{p}|={{t-\tdeg(u)+|\topp_p(u)|\choose {|\topp_p(u)|}}},
\end{eqnarray}
which is a polynomial in $t$ of degree $|\topp_p(u)|$.
Now, by (\ref{eqn_CardW}) and (\ref{eqn_Card[u]}),
\begin{eqnarray*}
|W|=\sum_{u\in W_p}{{t-\tdeg(u)+|\topp_p(u)|\choose {|\topp_p(u)|}}}.
\end{eqnarray*}
Let
\begin{eqnarray*}
h(x)=d\cdot\sum_{u\in W_p}{{x-\tdeg(u)+|\topp_p(u)|\choose {|\topp_p(u)|}}}.
\end{eqnarray*}
Then $h$ is a rational polynomial of degree $\max\{|\topp_p(u)|:u\in W_p\}$ such that $\HF_M(t)=h(t)$ for all $t\geq lp$.

(Uniqueness) Suppose $h'\in\QQ[x]$ and $\HF_M(t)=h'(t)$ for all $t\geq lp$.
Then $h-h'\in\QQ[x]$ and $h(t)-h'(t)=0$ for infinitely many $t$. Hence $h-h'=0$, or, $h=h'$.

\

(ii) It follows from part (i) and Lemma \ref{lemma_Hilbet_GK}.
\end{proof}

Theorem \ref{thm_GK=degHilbertPoly} together with Algorithm \ref{Alg_GB} gives an algorithm to compute the Gelfand-Kirillov dimension of a cyclic module over a \dd\ algebra over $R$.

\section{\gsbs\ for $A$-modules}
In this section, we develop the \gsb\ theory of \fg\ modules over a \dd\ algebra, which will be used to compute the Gelfand-Kirillov dimension of a \fg\ module over a \dd\ algebra in the next section.

Let $R$ be a finite field extension of $k$, $A=R[\s,\D;\sigma,\delta]$ be a \dd\ algebra of type $(m,n)$, and let $l=m+n$. 
Let $A^p$ ( $p\geq 1$) be the free left $A$-module of rank $p$ with the standard $A$-basis
$$\vect e_1=(1,0,0,\ldots,0),\vect e_2=(0,1,0,\ldots,0),\ldots,\vect e_p=(0,\ldots,0,1).$$
A \emph{monomial} in $A^p$ is an element of the form $\vect m=X^{\alpha}\vect e_i=x_1^{\alpha_1}\cdots x_l^{\alpha_l}\vect e_i$, where $\alpha=(\alpha_1,\ldots,\alpha_l)\in\NN^l$ and $1\leq i\leq p$.
The \emph{total degree} of $\vect m$ is defined as $\tdeg(\vect m)=\tdeg(X^{\alpha})=|\alpha|$.
Then the set
$\NNN=\{X^{\alpha}\vect e_i:\alpha\in \NN^l, 1\leq i\leq p\}$ of monomials in $A^p$ is an $R$-basis of $A^p$.
Thus every element $f\in A^p$ can be written in a unique way as an $R$-linear combination of monomials
$$
f=\sum_{i=1}^q c_i\vect m_i,\ \  0\neq c_i\in R, q\in \NN, \vect m_i\in \NNN.
$$
The \emph{total degree} of $f$ is defined as $\tdeg(f)=\max\{\tdeg(\vect m_i):1\leq i\leq q\}.$

We say $\vect m=X^{\alpha}\vect e_i$ is right (left, respectively) divisible by $\vect n=X^{\beta}\vect e_j$, $1\leq i,j\leq p$, if and only if $i=j$ and $X^{\alpha}$ is right (left, respectively) divisible by $X^{\beta}$, equivalently, if and only if $i=j$ and $\alpha_s\geq \beta_s$ for all $1\leq s\leq l$.
Suppose $f=\sum_{i=1}^q c_i\vect m_i\in A^p$ and $\vect n,\vect n_1,\ldots, \vect n_t\in \NNN$ for some $t\geq 1$. If each $\vect m_i$ is not right divisible by $\vect n$, then we say that $f$ is irreducible with respect to  $\vect n$. If $f$ is irreducible with respect to  every $\vect n_j\ (1\leq j\leq t)$ then we say that $f$ is irreducible with respect to  $\{\vect n_1,\ldots, \vect n_t\}$.

\begin{proposition}\label{prop_submodule_A^P_fg}
Every submodule of $A^p$ is \fg.
\end{proposition}
\begin{proof}
By Hilbert Basis Theorem \ref{thm_HibertBT}, $A$ is noetherian and thus so is $A^p$.
Hence every submodule of $A^p$ is \fg.
\end{proof}

As for monomials in $A$, we can similarly define \emph{leading monomial} $\lm(f)$, \emph{leading coefficient} $\lc(f)$, \emph{leading term} $\lt(f)$, \emph{left quotient} $\LQ{f}{g}$ and \emph{irreducible monomials} $\Irr(G)$ with respect to $G$, for $f\in A^p$, $g\in A$ and $G\subseteq A^p$.

\begin{definition}
A \emph{differential difference monomial ordering}, DD-monomial ordering for short, on $\NNN$ is a well-ordering $>$ on $\NNN$ such that
\begin{center}
if $\vect m>\vect n$
then $\lm(f \vect m)> \lm(f\vect n)$ for all $\vect  m,\vect n\in \NNN$ and ${0\neq f\in A}$.
\end{center}
\end{definition}

\begin{example}\label{exam_TOP_POT}\upshape
Let $>$ be a DD-monomial ordering on $\MM$ (recall that $\MM$ is the monomials of $A$). Then $>$ can be extended to a DD-monomial ordering on $\NNN$ as follows.

(1) We say $X^{\alpha}\vect e_i>_1X^{\beta}\vect e_j$ if and only if $X^{\alpha}>X^{\beta}$, or $X^{\alpha}=X^{\beta}$ and $i<j$. It is easy to see that $>_1$ is a DD-monomial ordering on $\NNN$. We call $>_1$ the TOP extension of $>$, where TOP stands for ``term over position'', following terminology in  \cite{adams1994introduction}.

(2) Similarly, we can introduce the POT (``position over term'') extension $>_2$ of $>$. Define $X^{\alpha}\vect e_i>_1X^{\beta}\vect e_j$ if and only if $i<j$ or ${i=j}$ and $X^{\alpha}>X^{\beta}$. It is easy to see that $>_2$ is also a DD-monomial ordering on $\NNN$.
\end{example}

Note that if the ordering $>$ in the above example is a total degree DD-monomial ordering then so is the TOP extension $>_1$.

\

The following lemma is similar to Lemma \ref{lemma_remainderUnique} (i) and it can be proved by induction on $\lm(f)$.
\begin{lemma}
\label{lemma_DivisionAlgorithm_A^p}
Let $>$ be a DD-monomial ordering on $\NNN$ and let $f_1,\ldots,f_q\in A^p$, $q\in\NN$.
Then every element $f\in A^p$ can be written as
$$
f=a_1f_1+\cdots+a_qf_q+r,
$$
where each $a_i\in A$, each $\lm(a_if_i)\leq \lm(f)$, $r\in A^p$ and either $r=0$ or $r$ is irreducible with respect to  $\{\lm(f_j):1\leq j\leq q\}$.
\end{lemma}

In the above lemma, $r$ is called a \emph{remainder} of $f$ on division by $\{f_1,\ldots,f_q\}$. We say that $f$ is reduced to $r$ by $\{f_1,\ldots,f_q\}$.

\begin{definition}\label{Def_GB_A^p}
Let $>$ be a DD-monomial ordering on  $\NNN$ and $M$ be a submodule of $A^p$.
A  subset $G\subseteq M$ is called a left \gsb\ for $M$ with respect to $>$ if, for any $0\neq f\in M$, there exists $g\in G$ such that $\lm(f)$ is right divisible by $\lm(g)$.
\end{definition}

Let $\vect m=X^{\alpha}\vect e_i$ and $\vect n=X^{\beta}\vect e_j$ be two monomials.
Define the \emph{least common left multiple} of $\vect m$ and $\vect n$ as
\[
\lclm(\vect m,\vect n)=\left\{
\begin{array}{ll}
  0   &  \mbox{if } i\neq j \\
  \lclm(X^{\alpha},X^{\beta})\ \vect e_i   &     \mbox{if } i=j
\end{array}
\right..
\]

Fix a monomial ordering on $A^p$. Suppose $f,g\in A^p$ and $\vect m=\lclm (\lm(f),\lm(g))$. Then the {\it S-vector} of $f$ and $g$ is defined as
\[
\Svect(f,g)=\LQ{\vect m}{{\lt}(f)}f-\LQ{\vect m}{\lt(g)}g.
\]
Then we have the following criterion for left \gsbs.
\begin{theorem}\label{thm_GSB_criterion-A^p}
Let $G$ be a  subset of $A^p$ and let $M$ be the submodule of $A^p$ generated by $G$. Then $G$ is a left \gsb\ for $M$ if and only if $\Svect(g_i,g_j)$ can be reduced to $0$ by $G$ for all $g_i,g_j\in G$.
\end{theorem}

The following theorem can be proved similarly to Theorem \ref{Thm_GB-linear+basis}.
\begin{theorem} \label{thm_R-basis_A^p/N}
Let $G$ be a \gsb\ for a submodule $M$ for $A^p$. Then {$\Irr(G)$} is an $R$-basis for the left $A$-module $A^p/M$.
\end{theorem}

\section{Gelfand-Kirillov dimension of finitely generated modules}
As in the previous section, let $A=R[\s,\D;\sigma,\delta]$ be a \dd\ algebra  of type $(m,n)$, where $R$ is a finite field extension (of $k$) of degree $d$ (suppose $V=\{v_1,\ldots,v_d\}$ is a $k$-basis of $R$). Let  $l=m+n$.
In this section, we use the \gsb\ theory of $A^p$ ($p\geq 1$) developed in the previous section to compute Gelfand-Kirillov dimension of a finitely generated left $A$-module.

Let $M$ be a \fg\ left $A$-module. Since every \fg\ left $A$-module is isomorphic to a quotient module of a \fg\ free left $A$-module, we may suppose that, throughout this section, $M=A^p/N$ for some $p\geq 1$ and some submodule $N$ of $A^p$. As in the previous section, denote the standard basis of $A^p$ by $\{\vect e_1,\ldots,\vect e_p\}$.
Let $E$ be the  $k$-subspace of $M$ generated by $\{\vect e_1+N,\ldots,\vect e_p+N\}$. Then $M=AE$.
Recall that $A$ has a natural grading:
$A=\oplus_{i\in \NN}A_i$, where $A_i$ is the $k$-subspace of $A$ spanned by $\MM_i=\{X^{\alpha}:\alpha\in\NN^l, |\alpha|=i\}$,
which  induces  a grading of $M=A^p/N$, i.e.,
$M=\oplus_{i\in\NN}M_i$, where $M_i=A_{i}E$ for $i\in \NN$.
Let $\NNN_i=\{u\vect e_j:u\in\MM_i, 1\leq j\leq p\}$ and $\NNN_{\leq i}=\NNN_0 \cup\cdots\cup \NNN_i$, $i\in \NN$.
Then, similar to Proposition \ref{prop_Hilbert_function}, we have the following proposition.

\begin{proposition}\label{prop_Hilbert_function_A^p}
Let $N$ be a submodule of $A^p$ and let $G$ be a left \gsb\ for $N $ with respect to a {total degree DD-monomial ordering}. Then the following hold:

(i) The set $B_i=\{\overline{vu}:v\in V, u\in\Irr(G)\cap\NNN_{\leq i}\}$ is a $k$-basis of $M_{\leq i}=M_0\oplus\cdots\oplus M_i$,  $i\in\NN$.

(ii) The Hilbert function of  $M$ is given by
$$
\HF_{M}(i)=d\cdot |\Irr(G)\cap \NNN_{\leq i} |, i\in \mathbb{N}.
$$

\end{proposition}

For a monomial $X^{\alpha}\vect e_i$ and $q\in\NN$, define $\topp_q(X^{\alpha}\vect e_i)=\topp_q(X^{\alpha})$ and $\sh_q(X^{\alpha}\vect e_i)=\sh_q(u)\vect e_i$.

The following theorem gives a relation between the Gelfand-Kirillov dimension and a finite \gsb\ for a \fg\ module over a \dd\ algebra.

\begin{theorem}\label{thm_GK=degHilbertPoly-A^p}
Let $R$ be a finite field extension of $k$ and $A=R[\s,\D;\sigma,\delta]$ be a \dd\ algebra.
Let $N$ be a submodule of the free $A$-module $A^p$ ($p\geq 1$), and $G$ be a left \gsb\ for $N$ with respect to a total degree DD-monomial ordering.
Denote the left $A$-module $A^p/N$ by $M$.
Set $$q=\max\{{\tdeg_{x_i}}(\lm(g)):g\in G, 1\leq i\leq l\}.$$
Then the following hold:

(i) There exists a unique polynomial $h\in\QQ[x]$ such that the Hilbert function $\HF_{M}$ of $M$ satisfies $\HF_{M}(t)=h(t)$ for all $t\geq lq$.

(ii) The Gelfand-Kirillov dimension of  $M$ is equal to the degree of $h$, which is given by
$$
\GK(M)=\deg h=\max\{|\topp_q(u)|:u\in \Irr(G)\cap \NNN_{\leq t}, \sh_q(u)=u\}, \ t\geq lq.
$$
\end{theorem}
\begin{proof}
(i) Let $t\in \NN$ and $W_p=\{u\in \Irr(G):\tdeg(u)\leq t, \sh_q(u)=u\}$.
Similar to Theorem \ref{thm_GK=degHilbertPoly}, one can show that
\[
\HF_{M}(t)=h(t)=pd\cdot \sum_{u\in W_p}{{t-\tdeg(u)+|\topp_q(u)|\choose {|\topp_q(u)|}}},
\]
which is a rational polynomial of degree $\max\{|\topp_q(u)|:u\in W_p\}$.

(ii) It follows from (i).
\end{proof}

Now we can easily write an algorithm from Theorem \ref{thm_GSB_criterion-A^p} and Theorem \ref{thm_GK=degHilbertPoly-A^p}.

\begin{algorithm}[GK-dimension of f.g. Modules over Differential Difference Algebras]
\upshape\

Input: $F=\{f_1,\ldots,f_s\}\subset A^p$, $s\in\NN$, each $f_i\neq 0$.

Output: $\GK(A^p/N)$, where $N$ is the left submodule of $A^p$ generated by $F$.

Initialization: $G:=F$, $P:=\{\{p,q\}:p\neq q, p,q\in F\}$.

WHILE $P\neq\emptyset$ DO

\hspace{12mm} Choose any pair $\{p,q\}\in P$, $P:=P\setminus \{\{p,q\}\}$

\hspace{12mm} $r:=$ a remainder of $\Svect(p,q)$ modulo $G$

\hspace{12mm} IF $r\neq 0$ THEN

\hspace{12mm} $G:=G\cup \{r\}$, $P:=P\cup \{\{g,r\}: g\in G\}$

END DO

$q:=\max\{{\tdeg_{x_i}}(\lm(g)):g\in G, 1\leq i\leq l\}$

Return $\max\{|\topp_q(u)|:u\in \Irr(G)\cap \NNN_{\leq t}, \sh_q(u)=u\}$

\end{algorithm}

\section*{Acknowledgements.} This work is supported in part by  the National Sciences and Engineering
Research Council (NSERC) of Canada and URGP from University of Manitoba.


\end{document}